\documentclass[12pt]{article}
\usepackage{amssymb,amsmath}
\usepackage{amsthm}

\def\NN{\mathbb N}
\def\RR{\mathbb R}
\def\FF{\mathcal F}
\def\sa{$\sigma$-algebra }

\def\d{W_1}
\def\et{\frac{\epsilon}{3}}
\def\enn{\nu_{N0}}
\def\emm{\nu_{mm}}
\def\fm{\FF_m}
\def\eh{\hat{\nu}}
\def\mn{\mu_N}
\def\mm{\mu_m}
\newtheorem{definition}{Definition}
\newtheorem{theorem}{Theorem}
\newtheorem{lemma}{Lemma}

\author{Gusti van Zyl\\ Department of Mathematics and Applied Mathematics \\ University of Pretoria}
\title{A measure approximation theorem for Wasserstein-robust expected values}

\begin{document}
	\maketitle 
\begin{abstract}
	We consider the problem of finding the infimum, over probability measures being in a ball defined by Wasserstein distance, of the expected value of a bounded Lipschitz random variable on $\RR^d$. We show that if the \sa is approximated in by a sequence of $\sigma$-algebras in a certain natural sense, then the solutions of the induced approximated minimization problems converge to that of the initial minimization problem.
\end{abstract}



\section{Introduction}
As part of the growing interest in models robust with respect to perturbations of the probability distribution, problems of the form
\begin{equation*}
\inf \left\{\int_X  V(x)d\nu(x):  d(\nu,\mu)\leq \theta \right\},
\end{equation*}
where $d$ denotes some distance on probability measures, are currently widely studied. In the financial literature this is connected with the problem of model risk, see for example \cite{doi:10.1111/mafi.12203}, \cite{feng2018model}, \cite{glasserman2014robust} and the references therein. For the related but more general problem of Distributionally Robust Stochastic Optimization, see for example the discussion in \cite{gao2016distributionally}. In this paper we restrict ourselves to the 1-Wasserstein distance. 

We study convergence for an approach that approximates the infinite dimensional problem with a finite dimensional problem, supported on some coarser $\sigma$-field $\FF_n$, which can then be solved with standard convex optimization algorithms. The purpose of this paper is to give general conditions on $V$ and some filtration $(\FF_n)$ that guarantees convergence. This requires amongst others a study of the effect of replacing the feasible set of the minimization problem with a ball centered on a measure $\mu_n$ that is defined on $\FF_n$.

Duality formulations of the optimization problem, or at least problems very similar to that studied in this paper, were obtained in \cite{ZHAO2018262} and \cite{MohajerinEsfahani2018}. For more general versions of these duality results, see again \cite{gao2016distributionally} and \cite{doi:10.1111/mafi.12203}. Although very powerful for obtaining explicit solutions in many cases, reducing the infinite dimensional to a finite dimensional problem, these duality results require the computation of a certain $\lambda c-$transform of $V$, see e.g. \cite{doi:10.1111/mafi.12203}. This transform is not always explicitly available or computationally trivial.

Another numerical approach which, exploits a duality representation in terms of a space of functions, is presented in \cite{Eckstein2019}, where the function space is approximated by a subspace generated by neural networks.

\section{Problem statement}
%We consider a Polish (i.e. separable and completely metrizable) space. 

Let $X=\RR^d,\ d\in\NN, $ equipped with the Euclidean distance and the Borel $\sigma-$algebra $\FF.$ As is well known, any finite distributional distribution generates a natural measure on $\RR^d$.
Let $\mathcal P_1=\mathcal P_1(X)$ be the set of probability measures on $(X,\FF)$ that have finite $1^{st}$ moment, equipped with the 1-Wasserstein metric $W_1$. As is also well known, see for example \cite[Chapter 6]{MR2459454}, convergence in the 1-Wasserstein metric is equivalent to weak convergence and convergence of $1^{st}$ moments.

Let $\mu\in \mathcal P_1$ and $V:\RR^d\to\RR$ a bounded Lipschitz function, with Lipschitz constant $K$, and $\theta>0$ a constant that expresses the extent of robustness required. (Such parameters are standard in the Robust Optimization literature, see for example \cite{gao2016distributionally} and the references therein.) As alluded to above, our minimization problem is to determine
\begin{equation}\label{eq:problem}
\inf \left\{\int_X  V(x)d\nu(x): \nu\in \mathcal P_1,\ \d(\nu,\mu)\leq \theta \right\}.
\end{equation} 

If $\mathcal G$ is a $\sigma-$algebra with $\mathcal G\subseteq \mathcal F$ then we will also refer to the restricted problem

\begin{equation}\label{eq:problem-m}
\inf \left\{\int_X  V(x)d\nu(x): \nu\in \mathcal P_1,\ \nu:\mathcal G\to[0,1], \ \d(\nu,\mu)\leq \theta \right\}.
\end{equation} 
{\bf Notation:} For any $\mu\in \mathcal P_1(X)$, $\sigma-$algebra $\mathcal G\subseteq \FF$ and $\theta>0$, ``Problem $(\mu,\mathcal G,\theta)$" refers to the minimization problem (\ref{eq:problem-m}). A {\it solution} of the problem will be a measure for which the infimum is attained; that is, a measure $\nu:\mathcal G\to [0,1]$ for which $F(\nu):=\int_X V(x)\ d\nu(x)$ equals the infimum in (\ref{eq:problem-m}) above.

By Kantorovich-Rubinstein duality \cite[Chapter 6]{MR2459454} we have that $F:\mathcal P_1(X)\to\RR$ satisfies $|F(\rho)-F(\sigma)|\leq K W_1(\rho,\sigma)$ for all $\rho,\sigma\in \mathcal P_1(X).$ In particular, $F$ is uniformly continuous on the metric space $(\mathcal P_1(X),W_1).$

%\section{Existence}
To show existence we apply a standard compactness argument to a minimizing sequence $(\nu_n)$, which exists since at least $\mu$ is in the feasible set of the minimization problem.

By the Prokhorov theorem, the sequence $(\nu_n)$ is weakly relatively compact if and only if it is tight. 
The tightness of bounded sets in $\mathcal P_1(\RR^d)$, or even $\mathcal P_p(\RR^d)$, is known in the literature, see for example \cite[Lemma 4.3]{MR2250166}. Hence $(\nu_n)$ has a subsequence that converges, to say $\nu$, weakly. Since $V$ is bounded and continuous, the weak convergence of the subsequence of measures ensures that the minimum is attained.

\section{Domain perturbation}

\begin{definition}
	We say that a filtration $(\FF_n)$ on $(X,\FF)$ is  approximating if for each measure $\rho:\FF\to [0,1]$ there exists a sequence of measures $ \rho_n:\FF_n\to [0,1]$ such that $\d(\rho_n, \rho)\to 0$ as $n\to\infty.$
\end{definition} 
Since the center of the ball $B(\mu,\theta)$ is changed when $\mu$ is replaced with a measure $\mu_n$ on a coarser \sa, we need an estimate on the change in minimum achieved if the domain is enlarged by a certain distance.

%, which should be implicit in the literature, for example, in \cite{MR1624098} where several perturbation of domain results are shown. 
Since we could not find a reference for a simple statement of the type needed, we give such an estimate in the lemma below. For a metric space $(X,d)$ we let $d(x,A):=\inf\{d(x,y):\ y\in A \}$ denote the distance from a point $x\in X$ to a set $A\subseteq X.$ 

\begin{lemma}\label{le:domain} Let $f$ be a uniformly continuous real-valued function on a metric space $(X,d)$. Let $A\subseteq X$ be a set on which $f$ has a maximum. Then the maximum $M_{t}:=\max\{ f(x):\ d(x,A)\leq t\}$ exists for $t>0$ small enough, and satisfies $M_{t}\downarrow M_0$ as $t\downarrow 0$. Similarly, if $f$ has a minimum on $A$ then the minimum $m_{t}:=\min\{ f(x):\ d(x,A)\leq t\}$ exists for $t>0$ small enough, and satisfies $m_{t}\uparrow m_0$ as $t\downarrow 0$. The convergence is uniform over $X$.
\end{lemma}

\begin{proof} 
	That $(M_t)$ is monotone is immediate.
	Now let $\epsilon>0$ be given. Let $t>0$ so that $d(x,y)\leq t\implies |f(x)-f(y)|\leq \epsilon$. 
	%	Now let $0<t\leq \delta.$ 
	Then  \begin{eqnarray*}
		d(x,A)\leq t \implies \exists y\in A,\ d(x,y)\leq t \\
		\implies f(x)\leq f(y)+\epsilon \leq M_0+\epsilon.
	\end{eqnarray*}
	Since $x$ is arbitrary we get $M_{t}\leq M_0+\epsilon.$ Since $t$ does not depend on $x\in X$, we have uniform convergence.
	
	The proof for the minimum $m_t$ is similar.
\end{proof}

\section{Convergence}
Our main result is the following. 
\begin{theorem}
	%Let $c:X\times Y\to \RPos$ be a finitely valued continuous cost function on compact Polish spaces $X,Y$. 
	%	Let $d$ denote the $1-$Wasserstein distance on the compact Polish space $X$. 
	Let $(\FF_n)$ be an approximating filtration, with $( \mu_n)$ adapted to $(\FF_n)$ such that $\displaystyle\lim_{n\rightarrow\infty} W_1(\mu_n,\mu)=0.$ 
	With $\nu$ denoting a solution of Problem $(\mu,\FF,\theta)$, and for each $m\in\NN,$ $\nu_m$ a solution of Problem $(\mu_m,\FF_m,\theta)$, we have that 
	$$\lim_{m\rightarrow\infty} F(\nu_m)=F(\nu).$$
\end{theorem}

\begin{proof}
	Let $\epsilon>0$ be given.  Adapting the notation of Lemma (\ref{le:domain}), let for any measure $\nu$ on $\FF$ and $\delta\geq 0$, $m_{\delta}(\nu)$ denote the minimum value for the problem $(\nu,\FF,\theta+\delta).$ 
	
	By Lemma \ref{le:domain} and the uniform continuity of $F$, there exists a $\delta>0$ so that  
	\begin{equation}\label{Eq:implication}
	W_1(\rho,\sigma)\leq \delta\implies \left(m_{\delta}(\rho)\leq m_0(\sigma)+\frac{\epsilon}{3} \text{ and } |F(\rho)-F(\sigma)|\leq \frac{\epsilon}{3}\right).\end{equation} 
		By the approximating property of $(\FF_n)$ there exists $N\in\NN$ such that $n\geq N\implies \d(\mu_n,\mu)\leq \delta.$ Let $\enn$ denote a solution of Problem $(\mu_N,\FF,\theta).$ It follows from Implication (\ref{Eq:implication}) that $|F(\enn)-F(\nu)|\leq \et.$ In other words, we have approximated the original problem with a problem where the midpoint is a measure on $\FF_N.$

Now let $m\geq N$ be large enough so that there exists a probability measure $\eh$ on $\FF_m$ with $\d(\eh,\enn)\leq \min\{\delta,\epsilon\}.$ Then
\begin{eqnarray*}
	\d(\eh,\mu_m)&\leq& \d(\eh,\enn)+\d(\enn,\mn)+\d(\mn,\mu)+\d(\mu,\mm) \\
	&\leq& \delta+\theta+\delta+\delta=\theta+3\delta.
\end{eqnarray*}
Hence $\eh$ is in the feasible set of Problem $(\mu_m,\fm,\theta+3\delta).$
Since the convergence in Lemma (\ref{le:domain}) is uniform over sets, we can use the lemma repeatedly to deduce that $F(\emm)$, where $\emm$ is a solution of Problem $(\mm,\fm,\theta),$ satisfies 
$$F(\emm)\leq F(\eh)+3\et=F(\eh)+\epsilon\leq F(\enn)+ 2  \epsilon\leq F(\nu)+3\epsilon.$$ (Note that Lemma (\ref{le:domain}) can be applied to $F|_{\fm}$, which is still uniformly continuous.) 
%To also bound $F(\emm)$ from below, let $\nu_{m0}$ denote a solution of Problem $(\mu_m,\FF,\theta).$ Note that $F(\nu_{m0})\leq F(\emm)$ since the latter minimization problem has a larger feasible set, and $F(\nu)\leq F(\nu_{m0})+\et$, since $\d(\mu_m,\mu)\leq \delta.$
%Combining these inequalities, we get
%\begin{eqnarray*}
%	F(\nu)-\epsilon\leq F(\nu_{m0})-\frac 2 3 \epsilon\leq F(\emm)-\frac 2 3 \epsilon\leq F(\eh)+\et\leq F(\enn)+\frac 2 3 \epsilon\leq F(\nu)+\epsilon.
%\end{eqnarray*}
Since $\epsilon>0$ is arbitrary, and the sequence $m\mapsto F(\nu_{mm})$ is monotone decreasing, this shows that $\displaystyle\lim_{m\to\infty}F(\emm)=F(\nu).$	
\end{proof}

\section{Acknowledgements} Financial support from the South African National Research Foundation, and the DST-Absa Directed Risk Research programme is gratefully acknowledged.

\bibliographystyle{plain}
\bibliography{robustdatabase}

%\begin{thebibliography}{00}

%% \bibitem[Author(year)]{label}
%% Text of bibliographic item

%\bibitem[ ()]{}

%\end{thebibliography}
\end{document}